\documentclass[reqno]{amsart}
\usepackage{amsmath,amssymb,amsthm}
\usepackage{tabularx,booktabs,multirow}
\usepackage{enumerate,enumitem}
\usepackage{color,graphicx,subfigure}

\DeclareGraphicsExtensions{.pdf,.png,.eps}

\newtheorem{theorem}{Theorem}
\newtheorem{lemma}[theorem]{Lemma}

\newtheorem{claim}{Claim}
\newtheoremstyle{definition}
  {4pt}
  {4pt}
  {\sl}
  {}
  {\bfseries}
  {.}
  {.5em}
  {}

\theoremstyle{definition}

\theoremstyle{remark}

\newtheoremstyle{introthms}
  {3pt}
  {3pt}
  {\itshape}
  {}
  {\bfseries}
  {.}
  {.5em}
  {\thmnote{#3}}
\theoremstyle{introthms}

\newcommand{\cC}{\mathcal{C}}
\newcommand{\cD}{\mathcal{D}}
\newcommand{\Nat}{\mathbb{N}}

\begin{document}
\title{Spectrum of mixed bi-uniform hypergraphs}

\author[M. Axenovich]{Maria Axenovich}
\address{Karlsruher Institut f\"ur Technologie, Karlsruhe, Germany
}
\email{maria.aksenovich@kit.edu} 
\author[E. Cherubini]{Enrica Cherubini}
\address{Karlsruher Institut f\"ur Technologie, Karlsruhe, Germany
}
\author[T. Ueckerdt]{Torsten Ueckerdt}
 \address{Karlsruher Institut f\"ur Technologie, Karlsruhe, Germany
}
\email{torsten.ueckerdt@kit.edu}

\date{\today}
\begin{abstract}
 A mixed hypergraph is a triple $H=(V,\cC,\cD)$, where $V$ is a set of vertices, $\cC$ and $\cD$ are sets of hyperedges. 
 A vertex-coloring of $H$ is proper if $C$-edges are not totally multicolored and  $D$-edges are not monochromatic.
 The feasible set $S(H)$ of $H$ is the set of all integers, $s$, such that $H$ has a proper coloring with $s$ colors.
 
 Bujt\'as and Tuza [\textit{Graphs and Combinatorics} \textbf{24} (2008), 1--12] 
 gave a characterization of  feasible sets for mixed hypergraphs with all $C$- and $D$-edges of the same size $r$,   $r\geq 3$.

 In this note, we give a short proof of a complete characterization of all possible feasible sets for mixed hypergraphs with all   $C$-edges  of size  $\ell$ and all $D$-edges of size   $m$, where $\ell, m \geq 2$.
 Moreover, we show that for every sequence $(r(s))_{s=\ell}^n$, $n \geq \ell$, of natural numbers there exists such a  hypergraph with exactly $r(s)$ proper colorings using  $s$ colors, $s = \ell,\ldots,n$, and no proper coloring with more than $n$ colors.
 Choosing $\ell = m=r$ this answers a question of Bujt\'as and Tuza, and generalizes their result with a shorter proof.
\end{abstract}

\maketitle


\section{Introduction}
 
A mixed hypergraph is a triple $H=(V,\cC,\cD)$, where $V$ is a set of vertices, $\cC$ is a set of subsets of vertices, called $C$-edges, and $\cD$ is a set of subsets of vertices, called $D$-edges.
A coloring $c: V \rightarrow \Nat$ of $H$ is proper if each $C$-edge has at least two vertices of the same color, and each $D$-edge has at least two vertices of distinct colors, i.e., $C$-edges are not rainbow and $D$-edges are not monochromatic.
An \emph{$s$-coloring} of $H$ is a coloring using exactly $s$ colors.
The notion of mixed hypergraphs was introduced by Voloshin~\cite{V}.
Their importance was emphasized by Kr\'al, who showed connections between feasible coloring of mixed hypergraphs and other classical coloring problems such as list-colorings of graphs, see~\cite{K2007}. 
One of the important parameters of a mixed hypergraph $H$, introduced in~\cite{V}, is its \emph{chromatic spectrum}, that is the sequence $(r_s(H))_{s=1}^n$, where $n = |V(H)|$, with
\[
 r_s(H) = \# \text{ non-isomorphic proper $s$-colorings of $H$}
\]
for $s = 1,\ldots,n$.
Here, two colorings are \emph{isomorphic} if they are the same up to relabeling the colors.
The set $\{s \in \Nat \mid r_s(H) \neq 0\}$ is called the \emph{feasible set} of $H$, i.e., it is the set of natural numbers
\[
 S(H)=\{s: \mbox{ there is a proper $s$-coloring of } H \}.
\]

Kr\'al proved that for any finite $S \subset \Nat$, $1 \notin S$, and any function $r:S \to \Nat \setminus \{0\}$ there is a mixed hypergraph $H$ with $S(H) = S$ and $r_s(H) = r(s)$ for every $s\in S$.
This extended a previous result by Jiang \textit{et al.}~\cite{J}.
Bujt\'as and Tuza~\cite{BT} described the feasible sets of \emph{uniform mixed hypergraphs}, i.e., where $C$- and $D$-edges have the same cardinality.
To state the results, we always shall assume that the mixed hypergraphs are non-empty, i.e., $\cC \cup \cD \neq \emptyset$.
We denote the interval of integers $\{i,i+1,\ldots,j\}$, $i \leq j$, as $[i,j]$.

\begin{theorem}[Bujt\'as, Tuza~\cite{BT}]{\ \\}
 Let $r\geq 3$ be an integer, and $S$ be a finite set of natural numbers. There is a non-empty  $r$-uniform mixed hypergraph $H=(V,\cC,\cD)$ with $S(H)=S$ if and only if
 \begin{enumerate}[label = (A\arabic*)]
  \item $S=[1,x]$, for some $x\geq r-1$, or\label{enum:interval-1}
  \item $\min(S) \geq r$, or\label{enum:min-r}
  \item $\min(S) = a \in [2,r-1]$, and $[a,r-1]\subseteq S$.\label{enum:general}
 \end{enumerate}
 
\end{theorem}

Zhao \textit{et al.}~\cite{ZDW} proved that if $r=3$ and $S$ is of type~\ref{enum:min-r} or~\ref{enum:general}, then for every function $r:S \to \Nat \setminus 0$ there is a $3$-uniform mixed hypergraph $H$ with $S(H) = S$ and $r_s(H) = r(s)$ for every $s \in S$.

A mixed hypergraph $H = (V,\cC,\cD)$ is \emph{$(\ell,m)$-uniform} if every $C$-edge has size $\ell$ and every $D$-edge has size $m$, i.e., $\cC \subseteq \binom{V}{\ell}$ and $\cD \subseteq \binom{V}{m}$. Here, we provide a complete characterization of feasible sets for $(\ell,m)$-uniform mixed hypergraphs for any $\ell, m \geq 2$. We also characterize the spectra of such hypergraphs in the range where at least $\ell$ colors are used.
  
\begin{theorem}\label{thm:main}
 Let $\ell,m \geq 2$ be integers, and $S$ be a finite set of natural numbers. There is a non-empty $(\ell,m)$-uniform mixed hypergraph $H = (V,\cC,\cD)$ with $S(H) = S$ if and only if
 \begin{enumerate}[label = (B\arabic*)]
  \item $S = [1,x]$ for some $x \geq \ell-1$, or\label{enum:new-interval-1}
  \item $\ell = 2$ and $S$ is an arbitrary interval, or\label{enum:new-interval}
  \item $\ell \geq 3$ and $\min(S) \geq \ell$, or\label{enum:new-min-ell}
  \item $\ell \geq 3$ and $\min(S) = a \in [2,\ell-1]$ and $[a,\ell-1]\subseteq S$.\label{enum:new-general}
 
 \end{enumerate}
 Moreover, if $r:S \to \Nat \setminus 0$ is a function, and $S$ is of type~\ref{enum:new-min-ell} or~\ref{enum:new-general}, then there is such a hypergraph $H$ with $S = S(H)$ and $r_s(H) = r(s)$ for every $s \in S$, $s \geq \ell$.
 
\end{theorem}

Note that in case~\ref{enum:new-min-ell} and~\ref{enum:new-general} the constructed hypergraph $H$ does not necessarily satisfy $r_s(H) = r(s)$ when $s < \ell$.
In fact, when $s < \ell$ one \emph{can not} prescribe the value of $r_s(H)$, as for example there is no $(\ell,m)$-uniform mixed hypergraph $H$ with $r_s(H) \geq 1$ and $r_{s+1}(H) = 1$ for $s+1 < \ell \leq |V(H)|$.
Indeed, taking any proper $s$-coloring and partitioning its largest color class in different ways into two non-empty parts, provides at least two different proper $(s+1)$-colorings.

We shall prove Theorem~\ref{thm:main} in Section~\ref{sec:proofs} below by following the same underlying approach as Bujt\'as and Tuza~\cite{BT},  Zhao \textit{et al.}~\cite{ZDW}, as well as Cherubini \cite{C}.
First, we provide an easy argument for the necessity of~\ref{enum:new-interval-1}--\ref{enum:new-general}.
Then we construct, for any $m \geq 2$, $\ell \geq 3$ and $1 < a < \ell \leq b$, an $(\ell,m)$-uniform hypergraph $H = H(\ell,m,a,b)$ with $S(H) = [a,\ell-1] \cup \{b\}$ and $r_b(H) = 1$.
Finally, we construct a more general hypergraph for any given feasible set of type~\ref{enum:new-min-ell} or~\ref{enum:new-general} and any given chromatic spectrum whose ``building blocks'' are the hypergraphs constructed before.

\section{Proofs}\label{sec:proofs}

We start with some easy observations.
  
\begin{lemma}\label{lem:easy-facts}
 Let $H = (V,\cC,\cD)$ be an $(\ell,m)$-uniform mixed hypergraph with at least $\ell-1$ vertices.
 \begin{itemize}[leftmargin=2em]
  \item If $1 \in S(H)$, then $S(H)=[1, x]$, $x\geq \ell-1$.
  \item If $\ell = 2$, then $S(H)$ is an interval.
  \item If $a \in S(H)$ for $2 \leq a \leq \ell-1$, then $[a,\ell-1] \subseteq S(H)$.  
 \end{itemize} 
\end{lemma}
\begin{proof}
  If $1 \in S(H)$, then $\cD = \emptyset$. Consider any proper $x$-coloring of $H$ with $x = \max(S(H))$. Observe that merging two color classes results again in a proper coloring of $H$, as $\cD = \emptyset$. Thus $S(H) = [1,x]$. On the other hand, any coloring with $\ell-1$ colors is a proper coloring of $H$. Thus $x \geq \ell-1$.
  
  If $\ell = 2$, then in every proper coloring of $H$ all vertices in the same connected component of the graph $G = (V,\cC)$ receive the same color. Hence we have $S(H) = S(H')$, where $H'$ is the hypergraph that arises from $H$ by contracting all vertices in the same connected component of $G$ into a single vertex and removing all $C$-edges. Since $H'$ has only $D$-edges, $S(H')$ is an interval. Note that $H'$, and hence $H$, is uncolorable if $H'$ contains a $D$-edge of size $1$.
  
 For $a \in S(H)$, $2 \leq a \leq \ell-1$, consider any proper $a$-coloring $c$ of $H$. If $a = \ell-1$ there is nothing to show. Otherwise, if we subdivide color classes in $c$ into new color classes so that the resulting number of colors is at most $\ell-1$, then the coloring is still proper.
  Indeed, each $C$-edge has $\ell$ vertices, thus at least two vertices of the same color. Each $D$-edge still uses distinct colors on its vertices.
  Such a subdivision exists whenever $H$ has at least $\ell-1$ vertices.
\end{proof}

\vskip 1cm

\subsection{Construction of a hypergraph $H = H(\ell,m,a,b)$ for $1<a<\ell \leq b$.}{\ \\}
Let $V = V(H) = X\times Y \times Z$, where $X=[a]$, $Y=[b]$ and $Z = [m]$.
We shall also refer to $\{x\}\times Y \times Z$ as a \emph{row}, and $X\times \{y\} \times Z$ as a \emph{column} of $H$.
The set $\cC$ of $C$-edges consists of all $\ell$-element sets containing at least $2$ vertices from the same column.
The set $\cD$ of $D$-edges consists of all $m$-element sets that are \emph{not} completely contained in any row and \emph{not} completely contained in any column.
  
We say that a coloring is a \emph{row-coloring} of $H$ if it assigns the same color to all vertices in each  row such that any two vertices from distinct rows get distinct colors.
We say that a coloring is a \emph{column-coloring} of $H$ if it assigns the same color to all vertices in each column such that any two vertices from distinct columns get distinct colors. 
Note that both row- and column-coloring are proper colorings of $H$.
Indeed, every $C$-edge contains two vertices from the same column by definition and two vertices from the same row since $a < \ell$.
Every $D$-edge contains two vertices from distinct rows and two vertices from distinct columns by definition.
We refer to Figure~\ref{fig:hypergraph1} for an illustration.

\begin{figure}[htb]
 \centering
 \subfigure[\label{fig:graph}]{
  \includegraphics{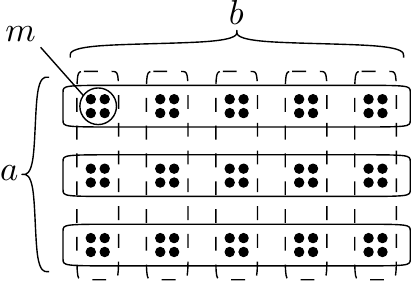}
 }
 \subfigure[\label{fig:row-coloring}]{
  \includegraphics{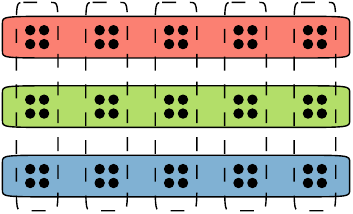}
 }
 \subfigure[\label{fig:column-coloring}]{
  \includegraphics{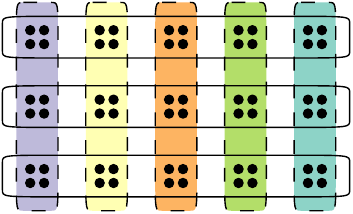}
 }
 \caption{Hypergraph $H(\ell,m,a,b)$ with $a = 3$, $b = 5$ and $m = 4$. Rows are highlighted in solid and columns in dashed. A row-coloring and a column-coloring are indicated in~\subref{fig:row-coloring} and~\subref{fig:column-coloring}, respectively.}
 \label{fig:hypergraph1}
\end{figure}

\begin{theorem}\label{thm:block}
 If $H = H(\ell,m,a,b)$, $1<a<\ell \leq b$, then $S(H) = [a,\ell-1] \cup \{b\}.$
 Moreover, every proper $a$-coloring of $H$ is a row-coloring and every proper $b$-coloring of $H$ is a column-coloring.
\end{theorem}

\begin{proof}
 First, we claim that $(V(H),\cD)$ is an $a$-chromatic hypergraph, and hence each proper coloring of $H$ uses at least $a$ colors.
 Indeed, if $H$ is colored with less than $a$ colors, then the biggest color class contains more than $b\cdot m$ vertices, and thus contains two vertices from distinct rows and columns. These two vertices together with some other $(m-2)$ vertices from that color class gives a monochromatic $D$-edge, a contradiction.
 
 Now observe that a row-coloring of $H$ uses exactly $a$ colors, which with Lemma~\ref{lem:easy-facts} implies that $[a,\ell-1] \subseteq S(H)$.
 Moreover, we claim that every proper coloring that uses exactly $a$ colors is a row-coloring.
 Indeed, we can argue by induction on $a$, with the case $a = 1$ being immediate.
 For $a \geq 2$, consider a proper $a$-coloring of $H$ and a largest color class $V'$. 
 Then $|V'| \geq bm$.
 If $V'$ contains two vertices from at least two different rows and different columns, these two vertices together with some other $(m-2)$ vertices from $V'$ give a monochromatic $D$-edge, a contradiction.
 Thus, $V'$ is contained in one fixed row and since $|V'| \geq bm$, it is actually equal to that row. Removing this row gives a proper $(a-1)$-coloring of $a-1$ rows, to which we can apply induction.
  
 A column-coloring of $H$ is a proper coloring using exactly $b$ colors. 
 In general, if each column of $H$ is monochromatic, then the columns must use distinct colors because of the $D$-edges.
 So, if all columns are monochromatic, then the coloring is a column-coloring and it uses exactly $b$ colors.
 If some column is not monochromatic and the total number of colors is at least $\ell$, then taking two vertices $v,v'$ of distinct colors from that column and $\ell-2$ vertices of distinct colors different from the colors of $v$ and $v'$ gives a rainbow $C$-edge, a contradiction. This concludes the proof.  
\end{proof}

\subsection{Construction of a hypergraph $H = H(\ell,m,a,b_1,b_2,\ldots,b_q)$, $1<a<\ell \leq b_1 \leq b_2 \leq \cdots \leq b_q$, and a coloring ${\bf c_i}(H)$ }\label{construction} {\ \\}
First consider the hypergraphs $H_1, H_2, \ldots, H_q$, where $H_i = H(\ell,m,a,b_i)$, $i=1,\ldots,q$, on pairwise disjoint sets of vertices.
If $q=1$, we define $H=H_1$. Otherwise, 
the hypergraph $H$ is the union of $H_i$'s, $i=1,\ldots,q$, and an additional set of $C$-edges, denoted $\cC_{add}$, where $\cC_{add} = \cC_1\cup \cC_2\cup \cC_3 \cup \cC_4$, is defined as below.
Let the vertex set of $H$ be $V$.
A subset of a row (respectively column) of some $H_i$, $1\leq i\leq q$, is called a \emph{transversal} of that row (respectively column) if it contains at most one vertex from every column (respectively row) of $H_i$.
We say that a triple of vertices $v, v', v'' \in V$ is \emph{special} in $H_i$, $1\leq i\leq q$, if $v$ and $v'$ are in the same row, and $v$ and $v''$ are in different rows but the same column of $H_i$.
We define $\cC_1,\cC_2,\cC_3,\cC_4$ implicitly by saying that for any $V' \in \binom{V}{\ell}$ we have
\begin{itemize}[leftmargin=2em]
 
 \setlength{\itemsep}{0.5em}
 \item $V' \in \cC_1$ if $V'$ contains a special triple of $H_i$ for some $1 \leq i \leq q$,
  
 \item $V' \in \cC_2$ if $V' = V_1 \dot\cup V_2$, $V_1$ is a transversal of size $a$ of the $1^{st}$ column of $H_i$ and $V_2 \subseteq V(H_{i+1}) \cup \cdots \cup V(H_q)$,   for some $1 \leq i \leq q-1$,

 \item $V' \in \cC_3$ if $V' = V_1 \dot\cup V_2$, $V_1$ is a transversal of size $\ell-2$ of the $1^{st}$ row of $H_i$ and $V_2$ consists of two vertices of the $1^{st}$ row of $H_{i+1}$, one of which has a column index that coincides with the column index of some vertex in $V_1$,  for some $1 \leq i \leq q-1$,   and

 \item $V' \in \cC_4$ if $V' = V_1 \dot\cup V_2$, $V_1$ is a transversal of size $\ell-2$ of the $k^{th}$ row of $H_i$  containing  a vertex from the $k^{th}$ column and $V_2$ consists of two vertices of the $1^{st}$ column of $H_j$,  one of which from the $k^{th}$ row of $H_j$,     for some $1 \leq i \leq q-1$, $1 \leq k \leq a$, $i < j \leq q$.
 
\end{itemize}

\begin{figure}[htb]
 \centering
 \includegraphics{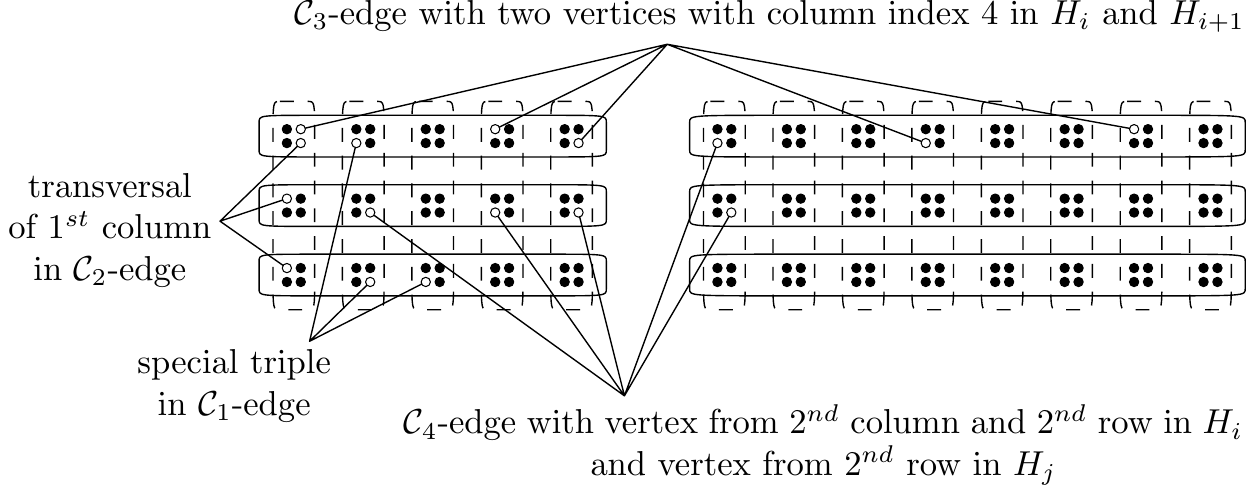}
 \caption{Illustration of $H(\ell,m,a,b_1,b_2)$ and one $\cC_1$-, $\cC_2$-, $\cC_3$- and $\cC_4$-edge, where $\ell = 5$, $m = 4$, $a=3$, $b_1 = 5$ and $b_2 = 8$.}
 \label{fig:C-add-edges}
\end{figure}

Now, we define a coloring ${\bf c_i}(H)$, $i=0,1,\ldots,q$.
If $i < q$ color $H_{i+1},\ldots,H_q$ with the row-colorings using color $k$ on the $k^{th}$ row, $k=1,\ldots,a$, and color the remaining hypergraphs (if $i \geq 1$) $H_1,\ldots,H_i$ with the column-colorings in which color $k$ is used on the $k^{th}$ column in each of these hypergraphs.
See Figure~\ref{fig:b_i-coloring}.

\begin{figure}[htb]
 \centering
 \includegraphics{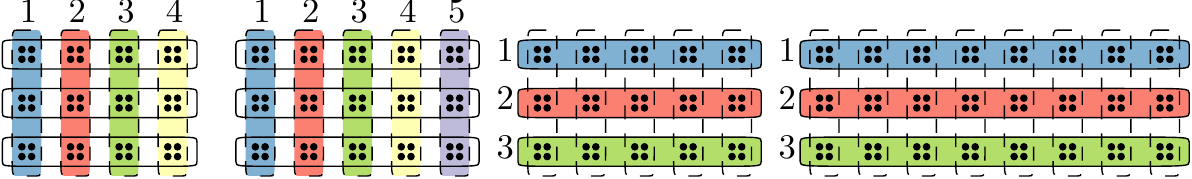}
 \caption{The coloring ${\bf c_2}(H)$ of $H(\ell,m,a,b_1,\ldots,b_q)$. Here $a = 3$, $\ell = 4$, $m = 6$, $q = 4$ and $(b_1,\ldots,b_q) = (4,5,5,8)$. Indices indicate the colors of corresponding rows or columns.}
 \label{fig:b_i-coloring}
\end{figure}

\begin{theorem}\label{thm:general}~\\
 For all $m \geq 2$ and $\ell \geq 3$ the $(\ell,m)$-uniform hypergraph $H= H(\ell,m,a,b_1,b_2,\ldots,b_q)$, $1 < a < \ell \leq b_1 \leq b_2 \leq \cdots \leq b_q$, has the feasible set $S(H) = [a,\ell-1] \cup \{b_1,b_2,\ldots,b_q\}.$
 
 Moreover, if $s = b_i$ for exactly $r(s)$ values of $i \in \{1,\ldots,q\}$, then $r_s(H) = r(s)$, $s \geq \ell$.
\end{theorem}

\begin{proof}
 Let $c$ be a proper coloring of $H$. We shall denote by $c(H_i)$ the set of colors used on $H_i$, and by $c(H)$  the set of colors used on $H$. We shall show the following.
 
 \begin{claim}\label{claim:big}
  One of the following must happen up-to relabeling of the colors:
  \begin{itemize}
   \item $|c(H)| \leq \ell-1$, or
   
   \item there is an index $i$ such that $c = {\bf c_i}(H)$.
   
  \end{itemize}
 \end{claim}
 
 \noindent 
 We prove Claim~\ref{claim:big} through a series of claims, one for each type of $C$-edges in $\cC_{add}$.
  
 \begin{claim}\label{claim:C1}
 Either $a< |c(H)|\leq \ell-1$ or each $H_i$ is row- or column-colored. 
 \end{claim}
 
 \noindent
 Assume that for some $i \in \{1,\ldots,q\}$, $H_i$ is neither row- nor column-colored, then we have to show that $|c(H)| \leq \ell-1$.
 From Theorem~\ref{thm:block} it follows immediately that $a < |c(H_i)|\leq \ell-1$.
 First, we shall show that there is a rainbow special triple in $H_i$.
 Indeed, since $ |c(H_i)|>a$ and $H_i$ has $a$ rows, there are two rows $R_1, R_2$ with at least $3$ colors on then. 
 If each column in $R_1\cup R_2$ were  monochromatic, then since $|c(H_i)|\leq \ell-1$, there would be  two monochromatic columns there in the same color, and thus 
 a monochromatic $D$-edge. Thus,  there are vertices    $v_1\in R_1$, $v_2 \in R_2$,  $v_1, v_2$ from the same column,  $c(v_1)\neq  c(v_2)$.
 Then a vertex $v$  of a third color in $R_1\cup R_2$  forms a special rainbow triple together with $v_1$ and $v_2$. 
If  $|c(H)| \geq \ell$ then $\{v, v_1, v_2\}$  together with $\ell-3$ vertices of different colors not appearing on this triple form a rainbow $\cC_1$-edge, a contradiction.
Thus $|c(H)|\leq \ell - 1$.  

 \medskip
 
 From now on, we assume that each $H_i$ is row- or column-colored.
 
 \begin{claim}\label{claim:C2}
   If $H_i$ is row-colored, then $|\bigcup_{j=i}^q c(H_j)| \leq \ell-1$ and each of $H_{i+1},\ldots,H_q$ is also row-colored.
 \end{claim} 
 
 \noindent
 Assume that $H_i$ is row-colored and $|\bigcup_{j=i}^q c(H_j)| \geq \ell$. Then any transversal $V_1$ of the $1^{st}$ column of $H_i$ together with $\ell-a$ vertices of $\bigcup_{j=i+1}^q H_j$ having distinct colors different from the colors in $V_1$ forms a rainbow $\cC_2$-edge, a contradiction.  Since a column-coloring uses at least $\ell$ colors,  each of $H_{i+1},\ldots,H_q$ is  not column-colored, thus it is  row-colored.
 
  \begin{claim}\label{claim:C3}
  If $H_i$ and $H_{i+1}$ are column-colored, $1 \leq i < q$, then the colors of the $k^{th}$ column of $H_i$ and $H_{i+1}$ coincide, $1 \leq k \leq b_i$.
 \end{claim}

 \noindent
 Assume not. Let $v,v' \in V(H_{i+1})$ be from the $1^{st}$ row, and the $k^{th}$, $k'^{th}$ column, respectively, $k\neq k'$. Let $V_1$ be a transversal of the $1^{st}$ row of $H_i$ of size $ \ell-2$, containing a vertex $v''$ from the $k^{th}$ column and using none of $c(v), c(v')$. Such $V_1$ exists, since $b_i \geq \ell$ and $c(v'')\neq c(v)$. Then $V_1 \dot\cup \{v,v'\}$ is a rainbow $\cC_3$-edge, a contradiction.

 \begin{claim}\label{claim:C4}
  If $H_i$ is column-colored and $H_j$ is row-colored, $1 \leq i < j \leq q$, then the colors of the $k^{th}$ column of $H_i$ and the $k^{th}$ row of $H_j$ coincide, $1 \leq k \leq a$.
 \end{claim}

 \noindent
 Assume not. Let $v,v'\in V(H_j)$ be from the $1^{st}$ column and the $k^{th}$, $k'^{th}$ rows, respectively, $k\neq k'$. Let $V_1$ be a transversal of size $\ell-2$ of the $k^{th}$ row of $H_i$ containing a vertex from the $k^{th}$ column, and using none of $c(v), c(v')$. Since $b_i \geq \ell$, such a transversal exists. Then $V_1 \dot\cup \{v,v'\}$ is a rainbow $\cC_4$-edge, a contradiction.
 
 \medskip
 
 To confirm Claim~\ref{claim:big},   observe that Claim~\ref{claim:C1}  implies that either $|c(H)|\leq \ell-1$ or each $H_i$ is row- or column-colored. 
 Claims~\ref{claim:C2}, ~\ref{claim:C3},~\ref{claim:C4} confirm  that if each $H_i$ is row- or column-colored, then $c= {\bf c_i}(H)$, for some index $i$.
 To finish the proof we need one last claim.

 \begin{claim}\label{claim:C5}
  The coloring ${\bf c_0}(H)$ is a proper $a$-coloring and the coloring ${\bf c_i}(H)$ is a proper $b_i$-coloring, $i=1,\ldots,q$.
 \end{claim}
 
 \noindent
 Indeed, every special triple contains two vertices from the same column and two vertices from the same row, and hence no $\cC_1$-edge is rainbow.
 No $\cC_2$-edge is rainbow, since it either contains two vertices from the same column in a column-colored $H_i$, or is completely contained in a subgraph that is $(\ell-1)$-colored.
 No $\cC_3$-edge is rainbow, since it contains two vertices from $H_i$ and $H_{i+1}$, $1 \leq i < q$, with the same column index and two vertices from $H_{i+1}$ from the same row.
 No $\cC_4$-edge is rainbow, since it contains a vertex $v$ from $H_i$ and a vertex $v'$ from $H_j$, $1 \leq i < j \leq q$, where row index and column index of $v$ and row index of $v'$ coincide, and two vertices from $H_j$ from the same column.
      
 \bigskip
 
 Continuing with the proof of Theorem~\ref{thm:general},
 recall that $|c(H)|\geq \chi((V,\cD)) \geq a$.
 Claim~\ref{claim:big} implies that $S(H) \subseteq [a,\ell-1] \cup \{b_1,b_2,\ldots,b_q\}.$
 Indeed, either the number of colors is at most $\ell-1$ or $c(H) = {\bf c_i }(H)$, $|{\bf c_i }(H)|=b_i$,  $i=1,\ldots,q$.
 This also implies that $r_s(H) \leq r(s)$ for every $s \in \{b_1,\ldots,b_q\}$.
 
 It remains to show that $S(H) \supseteq [a,\ell-1] \cup \{b_1,b_2,\ldots,b_q\}$ and that $r_s(H) \geq r(s)$ for every $s \in \{b_1,\ldots,b_q\}.$ 
 When $s=a$, consider the coloring ${\bf c_0}(H)$, which is a proper $a$-coloring by Claim~\ref{claim:C5}.
 When $s\in [a,\ell-1]$, Lemma~\ref{lem:easy-facts} implies the existence of a proper $s$-coloring of $H$.
 When $s \in \{b_1,\ldots, b_q\}$, consider a maximal block of repeated $s$, i.e., take largest $j$ and smallest $i$ such that $b_i=b_{i+1} = \cdots = b_j =s$.
 So, $r(s) = j-i+1$.
 Consider the colorings ${\bf c_i}(H),\ldots,{\bf c_j}(H)$.
 Claim~\ref{claim:C5} implies that these are $r(s)$ pairwise non-isomorphic proper $s$-colorings of $H$.
 Thus, $r_s(H) \geq r(s)$ for every $s \in \{b_1,\ldots, b_q\}$ and $S(H) \supseteq [a,\ell-1] \cup \{b_1,b_2,\ldots,b_q\}.$ 
\end{proof}

\subsection{Proof of Theorem~\ref{thm:main}}

Having Lemma~\ref{lem:easy-facts}, Theorem~\ref{thm:block} and Theorem~\ref{thm:general}, we can finally prove our main result.

\begin{proof}
 Lemma~\ref{lem:easy-facts} gives the necessity of~\ref{enum:new-interval-1}--\ref{enum:new-general}. Note that if $|V(H)|\leq \ell-1$, then 
 $H$ contains only $D$-edges and then $S(H)$ is an interval.   It remains to construct a hypergraph with the desired properties for any given $S$ of type~\ref{enum:new-interval-1},~\ref{enum:new-interval},~\ref{enum:new-min-ell}, or  ~\ref{enum:new-general}.
 
 \medskip
 
 Let $S$ be of type~\ref{enum:new-interval-1}, i.e., $S=[1,x]$, $x\geq \ell-1$.  Consider a hypergraph $H = (V,\cC,\cD)$ with $\cD = \emptyset$, $|V|= x+1$, $v,v'\in V$, and $\cC$ consisting of all $\ell$-subsets of $V$ containing $v$ and $v'$.
 In every $s$-coloring with $s > x$ the vertex set is totally multicolored, so no $C$-edge is properly colored. On the other hand, we obtain a proper $s$-coloring of $H$ for $s \leq x$ by distributing $s$ colors arbitrarily, so that $v$ and $v'$ receive the same color.

 \medskip
 
 Let  $\ell = 2$ and $S$ be of type~\ref{enum:new-interval}, i.e., $S = [a,b]$.  Define a $(2,m)$-uniform hypergraph $H $  as follows.
 The vertex set $V(H)$ consists of $b$ disjoint sets $V_1,\ldots,V_b$ of $m$ vertices each, and $\cC$ is the set of all pairs of vertices from the same $V_i$, $i=1,\ldots,b$.
 For $a>1$, we define $\cD$ to be the set of all $m$-subsets of $V(H)$ containing vertices from exactly two $V_i$'s  with $i \in \{1,\ldots,a\}$, and no vertex from $V_{a+1} \cup \cdots \cup V_b$.
 For $a=1$, we define $\cD = \emptyset$. Then it is easy to verify that $S(H)=[a,b]$.
 
 \medskip
 
 Let $\ell \geq 3$ and $S$ be of type~\ref{enum:new-min-ell}, i.e., $\min (S) \geq \ell$ and let $r : S \to \Nat \setminus 0$ be a given function. 
 Let $b_1\leq b_2\leq \cdots \leq b_q$ be nonnegative integers such that  $\{b_1, \ldots, b_q\} = S$ and each $b\in S$ be  repeated in the list exactly $r(b)$ times. 
 For example, when $S=\{3, 5\}$, $r(3)=2, r(5)= 1$, then $(b_1, b_2, b_3)=(3,3, 5)$.

 Let   $H$ be a graph formed by adding  an  extra $D$-edge  to the part $H_1$ of  $ H'= H(\ell,m,\ell-1,b_1,\ldots,b_q)$,     where this edge is contained in a row of    $H_1$,  but not in any of its  columns. 
  By Theorem~\ref{thm:general},  $H'$ has feasible set $\{\ell-1\} \cup S$ and there are exactly $r(s)$ proper $s$-colorings of $H'$ for every $s\in S$.
 By Theorem~\ref{thm:block},  every proper $(\ell-1)$-coloring of $H'$ is a row-coloring,  thus $H$ is no longer $(\ell-1)$-colorable because of the extra $D$-edge.
 On the other hand, the $b_i$-coloring ${\bf c_i}(H-\{D''\})$  is still a proper coloring of $H$,  $i=1, \ldots, q$, which implies that $S(H) = S$ and $r_s(H) = r(s)$ for every $s \in S$.
 
 \medskip
 
 Finally consider the case when $\ell \geq 3$ and $S$ is of type~\ref{enum:new-general}, i.e., $S = [a,\ell-1] \cup S'$ for some $a \in [2,\ell-1]$ and $S' = \emptyset$ or $\min(S') \geq \ell$. Let  $r : S \to \Nat \setminus 0$ be a given  function. For $S' \neq \emptyset$, consider the hypergraph $H = H(\ell,m,a,b_1,\ldots,b_q)$, where every $s \in S'$ appears exactly $r(s)$ times in the list $(b_1,\ldots,b_q)$. By Theorem~\ref{thm:general}, $H$ has the desired properties.   For $S' = \emptyset$,  let $H$ be obtained from   $H'= H(\ell,m,a,b)$  with some $b$,  $b \geq \ell$, by adding  an extra $D$-edge that
  is   contained in a column of $H'$, but not in a row of $H'$.
 Since by Theorem~\ref{thm:block}  every proper $b$-coloring of $H'$ is a column-coloring, $H$ is no longer $b$-colorable.
 On the other hand, the row-coloring using color $k$ on $k^{th}$ row, $k=1, \ldots, a$,  of $H'$ is a proper $a$-coloring of $H$, which implies that $S(H) = [a,\ell-1]$.
\end{proof}

\end{document}